\documentclass[12pt]{amsart}

\usepackage{amsmath}
\usepackage{amscd}
\usepackage{amssymb}
\usepackage{amsfonts}
\usepackage{amsthm}
\usepackage{enumerate}
\usepackage{hyperref}
\usepackage{color}
\usepackage{graphicx}

\theoremstyle{plain}
\newtheorem{thm}{Theorem}[section]

\newtheorem{prop}[thm]{Proposition}
\newtheorem{lem}[thm]{Lemma}
\newtheorem{cor}[thm]{Corollary}

\newtheorem{ques}[thm]{Question}

\theoremstyle{definition}
\newtheorem{dfn}[thm]{Definition}

\newtheorem{dfns-rems}[thm]{Definitions and Remarks}
\newtheorem{notas-rems}[thm]{Notations and Remarks}
\newtheorem{exmps-rems}[thm]{Examples and Remarks}

\DeclareMathOperator{\ind-match}{ind-match}

\makeatletter
\@namedef{subjclassname@2020}{%
  \textup{2020} Mathematics Subject Classification}
\makeatother

\begin{document}


\title[Regularity of squarefree powers of edge ideals]{On the Castelnuovo--Mumford regularity of squarefree powers of edge ideals}


\author[S. A. Seyed Fakhari]{S. A. Seyed Fakhari}

\address{S. A. Seyed Fakhari, School of Mathematics, Statistics and Computer Science,
College of Science, University of Tehran, Tehran, Iran.}

\email{aminfakhari@ut.ac.ir}


\begin{abstract}
Assume that $G$ is a graph with edge ideal $I(G)$ and matching number ${\rm match}(G)$. For every integer $s\geq 1$, we denote the $s$-th squarefree power of $I(G)$ by $I(G)^{[s]}$. It is shown that for every positive integer $s\leq {\rm match}(G)$, the inequality ${\rm reg}(I(G)^{[s]})\leq {\rm match}(G)+s$ holds  provided that $G$ belongs to either of the following classes: (i) very well-covered graphs, (ii) semi-Hamiltonian graphs, or (iii) sequentially Cohen-Macaulay graphs. Moreover, we prove that for every Cameron-Walker graph $G$ and for every positive integer $s\leq {\rm match}(G)$, we have ${\rm reg}(I(G)^{[s]})={\rm match}(G)+s$
\end{abstract}


\subjclass[2020]{Primary: 13D02, 05E40, 05C70}


\keywords{Castelnuovo--Mumford regularity, Edge ideal, Matching number, Squarefree power}


\maketitle


\section{Introduction} \label{sec1}

Let $\mathbb{K}$ be a field and $S = \mathbb{K}[x_1,\ldots,x_n]$  be the
polynomial ring in $n$ variables over $\mathbb{K}$. Suppose that $M$ is a graded $S$-module with minimal free resolution
$$0  \longrightarrow \cdots \longrightarrow  \bigoplus_{j}S(-j)^{\beta_{1,j}(M)} \longrightarrow \bigoplus_{j}S(-j)^{\beta_{0,j}(M)}   \longrightarrow  M \longrightarrow 0.$$
The integer $\beta_{i,j}(M)$ is called the $(i,j)$th graded Betti number of $M$. The Castelnuovo--Mumford regularity (or simply, regularity) of $M$,
denoted by ${\rm reg}(M)$, is defined as
$${\rm reg}(M)=\max\{j-i|\ \beta_{i,j}(M)\neq0\},$$
and it is an important invariant in commutative algebra and algebraic geometry.

There is a natural correspondence between quadratic squarefree monomial ideals of $S$ and finite simple graphs with $n$ vertices. To every simple graph $G$ with vertex set $V(G)=\big\{x_1, \ldots, x_n\big\}$ and edge set $E(G)$, we associate its {\it edge ideal} $I=I(G)$ defined by
$$I(G)=\big(x_ix_j: x_ix_j\in E(G)\big)\subseteq S.$$Computing and finding bounds for the regularity of edge ideals and their powers have been studied by a number of researchers (see for example \cite{ab}, \cite{b}, \cite{bbh}, \cite{bht}, \cite{dhs}, \cite{e2}, \cite{ha}, \cite{hh1}, \cite{js}, \cite{jns}, \cite{k}, \cite{msy}, \cite{sy}, \cite{sy2} and \cite{wo}).

In \cite{ehhs}, Erey, Herzog, Hibi and Saeedi Madani studied the {\it squarefree powers} of edge ideals. Recall that for a squarefree monomial ideal $I$, the $s$-th squarefree power of $I$, denoted by $I^{[s]}$ is the ideal generated by squarefree monomials belonging to $I^s$. Clearly, for an edge ideal $I(G)$, we have $I(G)^{[s]}=0$, for $s\geq {\rm match}(G)+1$, where ${\rm math}(G)$ denotes the matching number of $G$ which is the size of the largest matching in $G$. It is known by \cite[Theorem 6.7]{hv} that$${\rm reg}(I(G))\leq {\rm match}(G)+1.$$In \cite[Theorem 2.1]{ehhs}, it is proven that$${\rm reg}(I(G)^{[2]})\leq {\rm match}(G)+2.$$As a generalization of the above inequalities, Erey et al. \cite{ehhs} asked the following question.

\begin{ques} [\cite{ehhs}, Question 2.3] \label{quest}
Let $G$ be a graph. Is it true that for every positive integer $s\leq {\rm match}(G)$, the inequality
\[
\begin{array}{rl}
{\rm reg}(I(G)^{[s]})\leq {\rm match}(G)+s
\end{array} \tag{$\dagger$} \label{dag}
\]
holds?
\end{ques}
In \cite{bhz}, Bigdeli et al. proved that for any graph $G$, the ideal $I(G)^{[{\rm match(G)}]}$ has a linear resolution. In particular, inequality \ref{dag} is true for $s={\rm match}(G)$. When $G$ is a forest, Erey and Hibi \cite{eh} provided a sharp upper bound for ${\rm reg}(I(G)^{[s]})$ in terms
of the so-called $s$-admissable matching number of $G$. It follows from their result that inequality \ref{dag} is true for any forest.

The goal of this paper is to prove inequality \ref{dag} for several classes of graphs. More precisely, it is shown in Theorem \ref{n2} that for every graph $G$ and for each positive integer $s\leq {\rm match}(G)$,
\[
\begin{array}{rl}
{\rm reg}(I(G)^{[s]})\leq s+\lfloor n/2\rfloor.
\end{array} \tag{$\ddagger$} \label{ddag}
\]
As a consequence, we will see in Corollaries \ref{vwc} and \ref{shamil} that inequality \ref{dag} is true if $G$ is either a very well-covered or a semi-Hamiltonian graph. Moreover, we will see in Corollary \ref{nver} that inequality \ref{dag} also holds for every graph $G$ with at most nine vertices.

When $G$ is a bipartite graph, we prove a strengthened version of inequality \ref{ddag}. Indeed, we show in Theorem \ref{bip} that for any bipartite graph $G$ with bipartition $V(G)=X\cup Y$ and for every positive integer $s\leq {\rm match}(G)$,$${\rm reg}(I(G)^{[s]})\leq \min\{|X|, |Y|\}+s.$$As a consequence, inequality \ref{dag} is true for any sequentially Cohen-Macaulay bipartite graph (see Corollary \ref{seq}).

In Section \ref{sec4}, we compute the regularity of squarefree powers of edge ideals of Cameron-Walker graphs (see Section \ref{sec2} for the definition of Cameron-Walker graphs). As the main result of that section, we prove in Theorem \ref{cw} that for any Cameron-Walker graph and for every positive integer $s$ with $s\leq {\rm match}(G)$, we have $${\rm reg}(I(G)^{[s]})={\rm match}(G)+s.$$


\section{Preliminaries} \label{sec2}

In this section, we provide the definitions and basic facts which will be used in the next sections.

All graphs in this paper are simple, i.e., have no loops and no multiple edges. Let $G$ be a graph with vertex set $V(G)=\big\{x_1, \ldots,
x_n\big\}$ and edge set $E(G)$. We identify the vertices (resp. edges) of $G$ with variables (resp. corresponding quadratic monomials) of $S$. For a vertex $x_i$, the {\it neighbor set} of $x_i$ is $N_G(x_i)=\{x_j\mid x_ix_j\in E(G)\}$. We set $N_G[x_i]=N_G(x_i)\cup \{x_i\}$. The {\it degree} of $x_i$, denoted by ${\rm deg}_G(x_i)$ is the cardinality of $N_G(x_i)$. A vertex of degree one is called a {\it leaf}. An edge $e\in E(G)$ is a {\it pendant edge}, if it is incident to a leaf. A {\it pendant triangle} of $G$ is a triangle $T$ of $G$, with the property that exactly two vertices of $T$ have degree two in $G$. A {\it star triangle} is the graph consisting of finitely many triangles sharing exactly one vertex. For every subset $U\subset V(G)$, the graph $G\setminus U$ has vertex set $V(G\setminus U)=V(G)\setminus U$ and edge set $E(G\setminus U)=\{e\in E(G)\mid e\cap U=\emptyset\}$. A subgraph $H$ of $G$ is called {\it induced} provided that two vertices of $H$ are adjacent if and only if they are adjacent in $G$. A subset $C$ of $V(G)$ is called a {\it vertex cover} of $G$ if every edge of $G$ is incident to at least one vertex of $C$. A vertex cover $C$ is called a {\it minimal vertex cover} of $G$ if no proper subset of $C$ is a vertex cover of $G$. A graph $G$ without isolated vertices is said to be {\it very well-covered} if $|V(G)|$ is an even integer and every minimal vertex cover of $G$ has cardinality $|V(G)|/2$. A {\it Hamiltonian cycle} (resp. a {\it Hamiltonian path}) of $G$ is a cycle (resp. a path) which visits every vertex of $G$. The graph $G$ is a {\it Hamiltonian graph} if it has a Hamiltonian cycle. If $G$ has a Hamiltonian path, then we say that $G$ is a {\it semi-Hamiltonian graph}. In particular, every Hamiltonian graph is semi-Hamiltonian.

Let $G$ be a graph. A subset $M\subseteq E(G)$ is a {\it matching} if $e\cap e'=\emptyset$, for every pair of edges $e, e'\in M$. The cardinality of the largest matching of $G$ is called the {\it matching number} of $G$ and is denoted by ${\rm match}(G)$. If every vertex of $G$ is incident to an edge in $M$, then $M$ is a {\it perfect matching} of $G$. A matching $M$ of $G$ is an {\it induced matching} of $G$ if for every pair of edges $e, e'\in M$, there is no edge $f\in E(G)\setminus M$ with $f\subset e\cup e'$. The cardinality of the largest induced matching of $G$ is the {\it induced  matching number} of $G$ and is denoted by $\ind-match(G)$. It is clear that for every positive integer $s$, the ideal $I(G)^{[s]}$ is generated by monomials of the form $e_1\ldots e_s$, where $\{e_1, \ldots, e_s\}$ is a matching of $G$.

A graph is said to be a Cameron-Walker graph if ${\rm match}(G)=\ind-match(G)$. It is clear that a graph is Cameron-Walker if and only if all its connected components are Cameron-Walker. By \cite[Theorem 1]{cw} (see also \cite[Remark 0.1]{hhko}), a connected graph $G$ is a Cameron-Walker graph if and only if

$\bullet$ it is a star graph, or

$\bullet$ it is a star triangle, or

$\bullet$ it consists of a connected bipartite graph $H$ by vertex partition $V(H)=X\cup Y$ with the property that  there is at least one pendant edge attached  to each  vertex of $X$ and there may be some pendant triangles attached to each vertex of $Y$.

\begin{dfn}
Let $G$ be a graph. Two vertices $u$ and $v$ ($u$ may be equal to $v$) are said to be even-connected with respect to an $s$-fold product $e_1 \ldots e_s$ of edges of $G$, if there is an integer $r\geq 1$ and a sequence $p_0, p_1, \ldots, p_{2r+1}$ of vertices of $G$ such that the following conditions hold.
\begin{itemize}
\item[(i)] $p_0=u$ and $p_{2r+1}=v$.

\item[(ii)] $p_0p_1, p_1p_2, \ldots, p_{2r}p_{2r+1}$ are edges of $G$.

\item[(iii)] For all $0\leq k\leq r-1, \{p_{2k+1},p_{2k+2}\}=e_i$ for some $i$.

\item[(iv)] For all $i$, $\mid\{k\mid \{p_{2k+1},p_{2k+2}\}=e_i\}\mid \leq\mid\{j\mid e_i=e_j\}\mid$.
\end{itemize}
If the above conditions are satisfied, then we say that $u$ and $v$ are even-connected with respect to $e_1 \ldots e_s$. Moreover, the sequence $p_0, p_1, \ldots, p_{2r+1}$ is called an even-connection between $u$ and $v$ with respect to $e_1 \ldots e_s$.
\end{dfn}

Let $G$ be a graph and suppose that $e_1 \ldots e_s$ is an $s$-fold product of edges of $G$. Banerjee \cite[Theorems 61 and 6.7]{b} proved that $(I(G)^{s+1}:e_1 \ldots e_s)$ is generated by quadratic monomials $uv$ (it is possible that $u=v$) such that either $uv\in E(G)$ or $u$ and $v$ are are even-connected with respect to $e_1 \ldots e_s$.

Let $M$ be a finitely generated graded $S$-module and let $\beta_{i,j}(M)$ denote the $(i,j)$th graded Betti number of $M$. Then $M$ is said to
have a {\it linear resolution}, if for some integer $d$, $\beta_{i,i+t}(M)=0$
for all $i$ and every integer $t\neq d$.

Let $u$ be a monomial in $S$. The {\it support} of $u$, denoted by ${\rm supp}(u)$ is the set of variables dividing $u$. For a pair of monomials $u$ and $v$, the greatest common divisor of $u$ and $v$ will be denoted by ${\rm gcd}(u,v)$. If $I$ is monomial ideal, then $G(I)$ is the set of minimal monomial generators of $I$.


\section{Upper bound for the regularity of squarefree powers} \label{sec3}

In this section, we prove that inequality \ref{dag} is true for several classes of graphs. To this end, we determine some upper bounds for the regularity of squarefree powers of edge ideals, Theorems \ref{n2} and \ref{bip}. In order to prove these theorems, we first provide a strategy, inspired by Banerjee's idea \cite{b}, to bound the regularity of squarefree powers of edge ideals, Theorem \ref{regcol}.

We first need to find a suitable ordering for the minimal monomial generators of squarefree powers of edge ideals.

\begin{prop} \label{rfirst}
Assume that $G$ is a graph and $s\leq {\rm math}(G)-1$ is a positive integer. Then the monomials in $G(I(G)^{[s]})$ can be labeled as $u_1, \ldots, u_m$ such that
for every pair of integers $1\leq j< i\leq m$, one of the following conditions holds.
\begin{itemize}
\item [(i)] $(u_j:u_i) \subseteq (I(G)^{[s+1]}:u_i)$; or
\item [(ii)] there exists an integer $r\leq i-1$ such that $(u_r:u_i)$ is generated by a variable, and $(u_j:u_i)\subseteq (u_r:u_i)$.
\end{itemize}
\end{prop}

\begin{proof}
Using \cite[Theorem 4.12]{b}, the elements of $G(I(G)^s)$ can be labeled as $v_1, \ldots, v_t$ such that for every pair of integers $1\leq j< i\leq t$, one of the following conditions holds.
\begin{itemize}
\item [(1)] $(v_j:v_i) \subseteq (I(G)^{s+1}:v_i)$; or
\item [(2)] there exists an integer $k\leq i-1$ such that $(v_k:v_i)$ is generated by a variable, and $(v_j:v_i)\subseteq (v_k:v_i)$.
\end{itemize}

Since $G(I(G)^{[s]})\subseteq G(I(G)^s)$, there exist integers $\ell_1, \ldots, \ell_m$ such that $G(I(G)^{[s]})=\{v_{\ell_1}, \ldots, v_{\ell_m}\}$. For every integer $k$ with $1\leq k\leq m$, set $u_k:=v_{\ell_k}$. We claim that this labeling satisfies the desired property. To prove the claim, we fix integers $i$ and $j$ with $1\leq j< i\leq m$. Based on properties (1) and (2) above, we divide the rest of the proof into two cases.

\vspace{0.3cm}
{\bf Case 1.} Assume that $(v_{\ell_j}:v_{\ell_i}) \subseteq (I(G)^{s+1}:v_{\ell_i})$. Remind that that $v_{\ell_i}$ and $v_{\ell_j}$ are squarefree monomials. Therefore, $(v_{\ell_j}:v_{\ell_i})=(u)$, for some squarefree monomial $u$ with ${\rm gcd}(u, v_{\ell_i})=1$. Thus, $uv_{\ell_i}$ is a squarefree monomial and since $u\in (I(G)^{s+1}:v_{\ell_i})$, we conclude that $uv_{\ell_i}\in I(G)^{[s+1]}$. Consequently,$$(u_j:u_i)=(v_{\ell_j}:v_{\ell_i})=(u)\subseteq (I(G)^{[s+1]}: v_{\ell_i})=(I(G)^{[s+1]}: u_i).$$

\vspace{0.3cm}
{\bf Case 2.} Assume that there exists an integer $k\leq \ell_i-1$ such that $(v_k:v_{\ell_i})$ is generated by a variable, and $(v_{\ell_j}:v_{\ell_i})\subseteq (v_k:v_{\ell_i})$. Hence, $(v_k:v_{\ell_i})=(x_p)$, for some integer $p$ with $1\leq p\leq n$. It follows from the inclusion  $(v_{\ell_j}:v_{\ell_i})\subseteq (v_k:v_{\ell_i})$ that $x_p$ divides $v_{\ell_j}/{\rm gcd}(v_{\ell_j}, v_{\ell_i})$. Since, $v_{\ell_j}$ is a squarefree monomial, we deuce that $x_p$ does not divide $v_{\ell_i}$. As ${\rm deg}(v_k)={\rm deg}(v_{\ell_i})$, it follows from $(v_k:v_{\ell_i})=(x_p)$ that there is a variable $x_q$ dividing $v_{\ell_i}$ such that $v_k=x_pv_{\ell_i}/x_q$. This implies that $v_k$ is a squarefree monomial. Hence, $v_k=v_{\ell_r}=u_r$, for some integer $r$ with $1\leq r\leq m$. Using $k\leq \ell_i-1$, we have $\ell_r\leq \ell_i-1$. Therefore, $r\leq i-1$ and$$(u_j:u_i)\subseteq(u_r:u_i)=(x_p).$$ This completes the proof.
\end{proof}

Using Proposition \ref{rfirst}, we obtain the following result which provides a method to bound the regularity of squarefree powers of edge ideals.

\begin{thm} \label{regcol}
Assume that $G$ is a graph and $s\leq {\rm math}(G)-1$ is a positive integer. Let $G(I(G)^{[s]})=\{u_1, \ldots, u_m\}$ denote the set of minimal monomial generators of $I(G)^{[s]}$. Then$${\rm reg}(I(G)^{[s+1]})\leq \max\bigg\{{\rm reg}\big(I(G)^{[s+1]}:u_i\big)+2s, 1\leq i\leq m, {\rm reg}\big(I(G)^{[s]}\big)\bigg\}.$$
\end{thm}

\begin{proof}
Without loss of generality, we may assume that the labeling $u_1, \ldots, u_m$ of elements of $G(I(G)^{[s]})$ satisfies conditions (i) and (ii) of Proposition \ref{rfirst}. This implies that for every integer $i\geq 2$,
\begin{align*}
\big((I(G)^{[s+1]}, u_1, \ldots, u_{i-1}):u_i\big)=(I(G)^{[s+1]}:u_i)+({\rm some \ variables}).
\end{align*}
Hence, we conclude from \cite[Lemma 2.10]{b} that
\[
\begin{array}{rl}
{\rm reg}\big((I(G)^{[s+1]}, u_1, \ldots, u_{i-1}):u_i\big) \leq {\rm reg}(I(G)^{[s+1]}:u_i).
\end{array} \tag{1} \label{1}
\]

For every integer $i$ with $0\leq i\leq m$, set $I_i:=(I(G)^{[s+1]}, u_1, \ldots, u_i)$. In particular, $I_0=I(G)^{[s+1]}$ and $I_m=I(G)^{[s]}$. Consider the exact sequence
$$0\rightarrow S/(I_{i-1}:u_i)(-2s)\rightarrow S/I_{i-1}\rightarrow S/I_i\rightarrow 0,$$
for every $1\leq i\leq m$. It follows that$${\rm reg}(I_{i-1})\leq \max \big\{{\rm reg}(I_{i-1}:u_i)+2s, {\rm reg}(I_i)\big\}.$$Therefore,
\begin{align*}
& {\rm reg}(I(G)^{[s+1]})={\rm reg}(I_0)\leq \max\big\{{\rm reg}(I_{i-1}:u_i)+2s, 1\leq i\leq m, {\rm reg}(I_m)\big\}\\ & =\max\big\{{\rm reg}(I_{i-1}:u_i)+2s, 1\leq i\leq m, {\rm reg}(I(G)^{[s]})\big\}.
\end{align*}
The assertion now follows from inequality (\ref{1}).
\end{proof}

Using Theorem \ref{regcol}, in order to bound the regularity of squarefree powers of edge ideals, we need to study colon ideals of the form $(I(G)^{[s+1]}:u)$, where $u$ is a monomial in $G(I(G)^{[s]})$. In the following lemma, we show that these ideals are squarefree quadratic monomial ideals.

\begin{lem} \label{deg2}
Assume that $G$ is a graph and $s\leq {\rm math}(G)-1$ is a positive integer. Then for every monomial $u\in G(I(G)^{[s]})$, the ideal $(I(G)^{[s+1]}:u)$ is a squarefree monomial ideal generated in degree two.
\end{lem}

\begin{proof}
As $I(G)^{[s+1]}$ is a squarefree monomial ideal, $(I(G)^{[s+1]}:u)$ is a squarefree monomial ideal, too. Let $w$ be a squarefree monomial in the set of minimal monomial generators of $(I(G)^{[s+1]}:u)$. In particular, $uw$ is a squarefree monomial. Since$$w\in (I(G)^{[s+1]}:u)\subseteq (I(G)^{s+1}:u),$$ it follows from \cite[Theorem 6.1]{b} that there is a quadratic monomial $v\in (I(G)^{s+1}:u)$ which divides $w$. Since $uv$ divides $uw$, we deduce that $uv$ is a squarefree monomial and therefore, $v\in (I(G)^{[s+1]}:u)$. Thus, we conclude from $w\in G(I(G)^{[s+1]}:u)$ that $w=v$. Hence, $(I(G)^{[s+1]}:u)$ is a quadratic squarefree monomial ideal.
\end{proof}

The following corollary is a consequence of Lemma \ref{deg2} and determines the set of minimal monomial generators of the ideal $(I(G)^{[s+1]}:u)$.

\begin{cor} \label{graphh}
Let $G$ be a graph and $s\leq {\rm math}(G)-1$ be a positive integer. Also, let $u=e_1\ldots e_s$ be a monomial in  $G(I(G)^{[s]})$. Then there is a simple graph $H$ with vertex set $V(H)=V(G)\setminus {\rm supp}(u)$ such that $I(H)=(I(G)^{[s+1]}:u)$. Moreover, two vertices $x_i, x_j\in V(H)$ are adjacent in $H$ if and only if one of the following conditions holds.
\begin{itemize}
\item[(i)] $x_i$ and $x_j$ are adjacent in $G$; or
\item[(ii)] $x_i$ and $x_j$ are even-connected in $G$ with respect to $e_1\ldots e_s$.
\end{itemize}
\end{cor}

\begin{proof}
By Lemma \ref{deg2}, there is a graph $H$ with $I(H)=(I(G)^{[s+1]}:u)$. Since the variables in ${\rm supp}(u)$ do not divide the minimal monomial generators of the ideal $(I(G)^{[s+1]}:u)$, we have $V(H)=V(G)\setminus {\rm supp}(u)$ (where some of the vertices might be isolated). To determine the edges of $H$, assume that $x_i, x_j\in V(H)$ satisfy one of the conditions (i) and (ii) mentioned above. By \cite[Theorem 6.5]{b}, we have $ux_ix_j\in I(G)^{s+1}$. On the other hand, since $x_i, x_j\in V(H)=V(G)\setminus {\rm supp}(u)$, we conclude that $ux_ix_j$ is a squarefree monomial which implies that $x_ix_j\in (I(G)^{[s+1]}:u)$. This proves the "if" part.

To prove the "only if" part, suppose $x_i, x_j\in V(H)$ are adjacent in $H$ and assume that $x_ix_j\notin E(G)$. Since$$x_ix_j\in (I(G)^{[s+1]}:u)\subseteq (I(G)^{s+1}:u),$$we conclude from \cite[Theorem 6.7]{b} that $x_i$ and $x_j$ are even-connected in $G$ with respect to $e_1\ldots e_s$.
\end{proof}

We are now able to prove the first main result of this paper which provides a combinatorial upper bound for the regularity of squarefree powers of edge ideals.

\begin{thm} \label{n2}
Assume that $G$ is a graph with $n$ vertices and let $s\leq {\rm math}(G)$ be a positive integer. Then$${\rm reg}(I(G)^{[s]})\leq s+\lfloor n/2\rfloor.$$In particular, the answer of Question \ref{quest} is positive when $G$ has a matching of size $\lfloor n/2\rfloor$.
\end{thm}

\begin{proof}
We prove the assertion by induction on $s$. For $s=1$, we know from \cite[Theorem 6.7]{hv} that$${\rm reg}(I(G))\leq 1+{\rm match}(G)\leq 1+\lfloor n/2\rfloor.$$Thus, suppose $s\geq 2$. Let $G(I(G)^{[s-1]})=\{u_1, \ldots, u_m\}$ denote the set of minimal monomial generators of $I(G)^{[s-1]}$. It follows from Theorem \ref{regcol} that$${\rm reg}(I(G)^{[s]})\leq \max\bigg\{{\rm reg}\big(I(G)^{[s]}:u_i\big)+2(s-1), 1\leq i\leq m, {\rm reg}\big(I(G)^{[s-1]}\big)\bigg\}.$$Using the above inequality and the induction hypothesis, it is enough to prove that$${\rm reg}\big(I(G)^{[s]}:u_i\big)\leq \lfloor n/2\rfloor-s+2,$$for every integer $i$ with $1\leq i\leq m$. We conclude from Corollary \ref{graphh} that for every integer $i$ with $1\leq i\leq m$, there is a graph $H_i$ with $V(H_i)=V(G)\setminus {\rm supp}(u_i)$ such that $I(H_i)=(I(G)^{[s]}:u_i)$. In particular, every $H_i$ has $n-2(s-1)$ vertices. Therefore, we deduce from \cite[Theorem 6.7]{hv} that
\begin{align*}
& {\rm reg}\big(I(G)^{[s]}:u_i\big)\leq 1+{\rm match}(H_i)\leq 1+\bigg\lfloor\frac{|V(H_i)|}{2}\bigg\rfloor\\ & =1+\Big\lfloor\frac{n-2(s-1)}{2}\Big\rfloor=\Big\lfloor \frac{n}{2}\Big\rfloor-s+2.
\end{align*}
This completes the proof.
\end{proof}

As a consequence of Theorem \ref{n2}, we will see in the following corollaries that inequality \ref{dag} is true for very well-covered graph and for every semi-Hamiltonian graph.

\begin{cor} \label{vwc}
Let $G$ be a very well-covered graph. Then for every positive integer $s$ with $s\leq {\rm match}(G)$, we have $${\rm reg}(I(G)^{[s]})\leq {\rm match}(G)+s.$$
\end{cor}

\begin{proof}
We know from \cite[Theorem 1.2]{f2} that every very well-covered graph has a perfect matching. Thus, the assertion follows from Theorem \ref{n2}.
\end{proof}

\begin{cor} \label{shamil}
Let $G$ be a semi-Hamiltonian graph. Then for every positive integer $s$ with $s\leq {\rm match}(G)$, we have $${\rm reg}(I(G)^{[s]})\leq {\rm match}(G)+s.$$
\end{cor}

\begin{proof}
Suppose $V(G)=\{x_1, \ldots, x_n\}$ is the vertex set of $G$. Without loss of generality, we may assume that $x_1, x_2, \ldots, x_n$ is a Hamiltonian path of $G$.

$\bullet$ If $n$ is even, then the set $\{x_1x_2, x_3x_4, \ldots, x_{n-1}x_n\}$ of edges of $G$ form a matching of size $n/2$.

$\bullet$ If $n$ is odd, then the the set $\{x_1x_2, x_3x_4, \ldots, x_{n-2}x_{n-1}\}$ of edges of $G$ form a matching of size $(n-1)/2$ in $G$.

In both cases $G$ has a matching of size $\lfloor n/2\rfloor$. Hence, the assertion follows from Theorem \ref{n2}.
\end{proof}

The following corollary shows that inequality \ref{dag} is true for every graph with at most nine vertices.

\begin{cor} \label{nver}
Let $G$ be a graph with at most nine vertices. Then for every positive integer $s$ with $s\leq {\rm match}(G)$, we have $${\rm reg}(I(G)^{[s]})\leq {\rm match}(G)+s.$$
\end{cor}

\begin{proof}
For $s=1$, the above inequality follows from \cite[Theorem 6.7]{hv}. For $s=2$, the assertion is known by \cite[Theorem 2.11]{ehhs}. Also, for $s={\rm match}(G)$, the above inequality is known by \cite[Theorem 5.1]{bhz}. So, there is nothing to prove if ${\rm match}(G)\leq 3$. Consequently, suppose that $|V(G)|\in\{8,9\}$ and ${\rm match}(G)=4$. In this case the assertion follows from Theorem \ref{n2}.
\end{proof}

When $G$ is a bipartite graph, we are able to improve the inequality obtained in Theorem \ref{n2}.

\begin{thm} \label{bip}
Let $G$ be a bipartite graph and suppose that $V(G)=X\cup Y$ is a bipartition for the vertex set of $G$. Then for every positive integer $s$ with $s\leq {\rm match}(G)$, we have $${\rm reg}(I(G)^{[s]})\leq \min\{|X|, |Y|\}+s.$$
\end{thm}

\begin{proof}
We prove the assertion by induction on $s$. For $s=1$, we know from \cite[Theorem 6.7]{hv} that$${\rm reg}(I(G))\leq 1+{\rm match}(G)\leq 1+\min\{|X|, |Y|\}.$$Thus, suppose $s\geq 2$. Let $G(I(G)^{[s-1]})=\{u_1, \ldots, u_m\}$ denote the set of minimal monomial generators of $I(G)^{[s-1]}$. It follows from Theorem \ref{regcol} that$${\rm reg}(I(G)^{[s]})\leq \max\bigg\{{\rm reg}\big(I(G)^{[s]}:u_i\big)+2(s-1), 1\leq i\leq m, {\rm reg}\big(I(G)^{[s-1]}\big)\bigg\}.$$Using the above inequality and the induction hypothesis, it is enough to prove that$${\rm reg}\big(I(G)^{[s]}:u_i\big)\leq \min\{|X|, |Y|\}-s+2,$$for every integer $i$ with $1\leq i\leq m$. We conclude from Corollary \ref{graphh} that for every integer $i$ with $1\leq i\leq m$, there is a graph $H_i$ with $V(H_i)=V(G)\setminus {\rm supp}(u_i)$ such that $I(H_i)=(I(G)^{[s]}:u_i)$. Set $X_i:=X\setminus {\rm supp}(u_i)$ and $Y_i:=Y\setminus {\rm supp}(u_i)$. As $u_i$ is the product of $s-1$ disjoint edges of $G$, we have$$|X\cap {\rm supp}(u_i)|=|Y\cap {\rm supp}(u_i)|=s-1.$$Consequently$$|X_i|=|X|-(s-1) \ \ \ {\rm and} \ \ \ |Y_i|=|Y|-(s-1).$$Since $G$ is a bipartite graph, it easily follows from the definition of even-connection that two distinct vertices of $X$ can not be even-connected with respect to $u_i$. Similarly, two distinct vertices of $Y$ can not be even-connected with respect to $u_i$. This means that $H_i$ is a bipartite graph and $V(H_i)=X_i\cup Y_i$ is a bipartition for its vertex set. Therefore, we deduce from \cite[Theorem 6.7]{hv} that
\begin{align*}
& {\rm reg}\big(I(G)^{[s]}:u_i\big)\leq 1+{\rm match}(H_i)\leq 1+\min\{|X_i|, |Y_i|\}\\ & =1+\min\{|X|, |Y|\}-(s-1)\\ &=\min\{|X|, |Y|\}-s+2.
\end{align*}
This completes the proof.
\end{proof}

Recall that a graph $G$ is a {\it sequentially Cohen-Macaulay graph} if the ring $S/I(G)$ has the same property. The following corollary shows that inequality \ref{dag} is true for any sequentially Cohen-Macaulay bipartite graph.

\begin{cor} \label{seq}
Let $G$ be a sequentially Cohen-Macaulay bipartite graph. Then for every positive integer $s$ with $s\leq {\rm match}(G)$, we have $${\rm reg}(I(G)^{[s]})\leq {\rm match}(G)+s.$$
\end{cor}

\begin{proof}
Let $V(G)=X\cup Y$ be a bipartition for the vertex set of $G$. Using induction on $|V(G)|$, we prove that$${\rm match}(G)=\min\{|X|, |Y|\}.$$Then the assertion follows from Theorem \ref{bip}.

To prove the claim, it follows from \cite[Corollary 3.11]{vv} that $G$ has a vertex $x$ of degree one such that $G\setminus N_G[x]$ is sequentially Cohen-Macaulay. Let $y$ be the unique unique neighbor of $x$. We deduce from the induction hypothesis that $G\setminus N_G[x]$ has a matching of size $\min\{|X|, |Y|\}-1$. This matching together with the edge $xy$ forms a matching of size $\min\{|X|, |Y|\}$ in $G$.
\end{proof}


\section{Cameron-Walker graphs} \label{sec4}

As the main result of this section, we compute the regularity of squarefree powers of edge ideals of Cameron-Walker graphs, Theorem \ref{cw}. We first need the following simple lemmas. In the first lemma, we determine the matching number of Cameron-Walker bipartite graphs.

\begin{lem} \label{cwmatch}
Let $G$ be a Cameron-Walker bipartite graph and assume that $V(G)=X\cup Y$ is a bipartition for the vertex set of $G$. Then$${\rm match}(G)=\min\{|X|, |Y|\}.$$
\end{lem}

\begin{proof}
Without loss of generality, we may suppose that $G$ is a connected graph. Then the claim easily follows from the structure of Cameron-Walker connected graphs, mentioned in Section \ref{sec2}.
\end{proof}

The following lemma helps us to use induction for computing the regularity of squarefree powers of edge ideals of Cameron-Walker graphs.

\begin{lem} \label{cwcolon}
Let $G$ be a graph and assume that $T$ is a triangle of $G$, with vertex set $V(T)=\{x, y, z\}$. Suppose that ${\rm deg}_G(x)={\rm deg}_G(y)=2$. Set $H:=G\setminus \{x, y\}$. Then for every integer $s\geq 2$,$$\big(I(G)^{[s]}: xy\big)=I(H)^{[s-1]}.$$
\end{lem}

\begin{proof}
The inclusion "$\supseteq$" is trivial. To prove that reverse inclusion, let $u$ be a monomial in the set of minimal monomial generators of $(I(G)^{[s]}: xy)$. Then $uxy$ is a squarefree monomial and there exist disjoint edges $e_1, \ldots , e_s\in E(G)$ such that $e_1\ldots e_s$ divides $uxy$. If either $x$ or $y$ does not divide $e_1\ldots e_s$, then clearly, $u\in I(H)^{[s-1]}$. So, suppose that $x$ and $y$ divide $e_1\ldots e_s$. If there is an integer $k$ with $1\leq k\leq s$ such that $e_k=xy$, then $e_1\ldots e_{k-1}e_{k+1}\ldots e_s$ divides $u$. Thus, $u\in I(H)^{[s-1]}$. Consequently, we assume that for every integer $k$ with $1\leq k\leq s$, we have $e_k\neq xy$. This yields that $x$ and $y$ appear in distinct edges $e_i$ and $e_j$ with $1\leq i, j\leq s$. Both of these edges must be incident to $z$ which is a contradiction, as the edges $e_1, \ldots , e_s$ are disjoint.
\end{proof}

We are now ready to prove the main result of this section.

\begin{thm} \label{cw}
Let $G$ be a Cameron-Walker graph. Then for every positive integer $s$ with $s\leq {\rm match}(G)$, we have $${\rm reg}(I(G)^{[s]})={\rm match}(G)+s.$$
\end{thm}

\begin{proof}
It follows from \cite[Theorem 2.1]{ehhs} that for every positive integer $s\leq {\rm match}(G)$,$${\rm reg}(I(G)^{[s]})\geq \ind-match(G)+s={\rm match}(G)+s.$$Therefore, it is enough to prove that$${\rm reg}(I(G)^{[s]})\leq {\rm match}(G)+s$$for every positive integer $s\leq {\rm match}(G)$. We use induction on $|E(G)|$. If $G$ is bipartite, then the above inequality follows from Theorem \ref{bip} and Lemma \ref{cwmatch}. So, suppose $G$ is not a bipartite graph. In particular, it follows from the construction of Cameron-Walker graphs, mentioned in Section \ref{sec2}, that $G$ contains a triangle $T$ with vertex set $V(T):=\{x, y, z\}$ such that ${\rm deg}_G(x)={\rm deg}_G(y)=2$. Using \cite[Theorem 6.7]{hv}, we may assume that $s\geq 2$. Consider the following short exact sequence.
\begin{align*}
0 \longrightarrow \frac{S}{(I(G)^{[s]}:xy)}(-2)\longrightarrow \frac{S}{I(G)^{[s]}}\longrightarrow \frac{S}{I(G)^{[s]}+(xy)}\longrightarrow 0
\end{align*}
Let $H_1$ be the graph which is obtained from $G$ by deleting the edge $xy$. Note that$$I(G)^{[s]}+(xy)=I(H_1)^{[s]}+(xy).$$Set $H_2:=G\setminus \{x, y\}$. It follows from Lemma \ref{cwcolon} and the above exact sequence that
\[
\begin{array}{rl}
{\rm reg}(I(G)^{[s]})\leq \max\big\{{\rm reg}(I(H_2)^{[s-1]})+2, {\rm reg}(I(H_1)^{[s]},xy)\big\}.
\end{array} \tag{2} \label{2}
\]
It is obvious from the structure of Cameron-Walker graphs that $H_2$ is a Cameron-Walker graph. Moreover, ${\rm match}(H_2)={\rm match}(G)-1$. Therefore, we deduce from the induction hypothesis that
\[
\begin{array}{rl}
{\rm reg}(I(H_2)^{[s-1]})\leq {\rm match}(H_2)+s-1={\rm match}(G)+s-2.
\end{array} \tag{3} \label{3}
\]

Now, consider the following short exact sequence.
\begin{align*}
0 \longrightarrow \frac{S}{\big((I(H_1)^{[s]},xy):xz\big)}(-2)\longrightarrow \frac{S}{(I(H_1)^{[s]},xy)}\longrightarrow \frac{S}{(I(H_1)^{[s]}, xy,xz)}\longrightarrow 0
\end{align*}
Let $H_3$ be the graph obtained from $H_1$ by deleting the edge $xz$ and note that$$(I(H_1)^{[s]}, xy,xz)=(I(H_3)^{[s]}, xy,xz).$$Set $H_4:=H_1\setminus \{x,y,z\}$. Clearly, $xz$ is a pendant edge of $H_1$. Hence, we conclude from \cite[Lemma 22]{eh} that
\begin{align*}
& \big((I(H_1)^{[s]},xy):xz\big)=(I(H_1)^{[s]}:xz)+(xy:xz)\\ &=I(H_1\setminus \{x,z\})^{[s-1]}+(y)=I(H_4)^{[s-1]}+(y).
\end{align*}
Thus, it follows from the above exact sequence that
\[
\begin{array}{rl}
{\rm reg}(I(H_1)^{[s]},xy)\leq \max\big\{{\rm reg}(I(H_4)^{[s-1]}+(y))+2, {\rm reg}(I(H_3)^{[s]},xy,xz)\big\}.
\end{array} \tag{4} \label{4}
\]
It is easy to see that $H_4$ is a Cameron-Walker graph and ${\rm match}(H_4)={\rm match}(G)-1$. Therefore, using \cite[Theorem 20.2]{p'} and the induction hypothesis, we have
\[
\begin{array}{rl}
{\rm reg}\big(I(H_4)^{[s-1]}+(y)\big)\leq {\rm match}(H_4)+s-1={\rm match}(G)+s-2.
\end{array} \tag{5} \label{5}
\]

Consider the following short exact sequence.
\begin{align*}
0 & \longrightarrow \frac{S}{\big((I(H_3)^{[s]},xy,xz):yz\big)}(-2)\longrightarrow \frac{S}{(I(H_3)^{[s]},xy,xz)}\\ &\longrightarrow \frac{S}{(I(H_3)^{[s]}, xy,xz, yz)}\longrightarrow 0
\end{align*}
Let $H_5$ be the graph obtained from $H_3$ by deleting the edge $yz$ and note that$$(I(H_3)^{[s]}, xy,xz,yz)=(I(H_5)^{[s]}, xy,xz,yz).$$Clearly, $yz$ is a pendant edge of $H_3$. Hence, we conclude from \cite[Lemma 22]{eh} that
\begin{align*}
& \big((I(H_3)^{[s]},xy,xz):yz\big)=(I(H_3)^{[s]}:yz)+\big((xy, xz):yz\big)\\ &=I(H_3\setminus \{y,z\})^{[s-1]}+(x)=I(H_4)^{[s-1]}+(x).
\end{align*}
Thus, it follows from the above exact sequence that
\[
\begin{array}{rl}
{\rm reg}(I(H_3)^{[s]},xy,xz)\leq\max\big\{{\rm reg}(I(H_4)^{[s-1]}+(x))+2, {\rm reg}(I(H_5)^{[s]},xy,xz,yz)\big\}.
\end{array} \tag{6} \label{6}
\]
Remind that $H_4$ is a Cameron-Walker graph with ${\rm match}(H_4)={\rm match}(G)-1$. Therefore, we conclude from \cite[Theorem 20.2]{p'} and the induction hypothesis that
\[
\begin{array}{rl}
{\rm reg}\big(I(H_4)^{[s-1]}+(x)\big)\leq {\rm match}(H_4)+s-1={\rm match}(G)+s-2.
\end{array} \tag{7} \label{7}
\]
Note that $H_5$ is a Cameron-Walker graph with ${\rm match}(H_5)={\rm match}(G)-1$. Hence, using \cite[Corollary 3.2]{he} (see also \cite[Theorem 1.2]{km}) and the induction hypothesis, we have
\begin{align*}
& {\rm reg}(I(H_5)^{[s]},xy,xz,yz)\leq {\rm reg}(I(H_5)^{[s]})+{\rm reg}(xy,xz,yz)-1\\ &\leq {\rm match}(H_5)+s+2-1={\rm match}(G)-1+s+1\\ &={\rm match}(G)+s.
\end{align*}
The assertion follows by combining the above inequality with inequalities (\ref{2}), (\ref{3}), (\ref{4}), (\ref{5}), (\ref{6}) and (\ref{7}).
\end{proof}

The following corollary is an immediate consequence of Theorem \ref{cw}.

\begin{cor}
Let $G$ be a Cameron-Walker graph and suppose that $s\leq {\rm match}(G)$ is a positive integer. Then $I(G)^{[s]}$ has a linear resolution if and only if $s={\rm match}(G)$.
\end{cor}






\begin{thebibliography}{10}

\bibitem {ab} A. Alilooee , A. Banerjee, Powers of edge ideals of regularity three bipartite graphs, {\it J. Commut. Algebra}, {\bf 9} (2017), 441--454.

\bibitem {b} A. Banerjee, The regularity of powers of edge ideals, {\it J. Algebraic Combin.} {\bf 41} (2015), 303--321.

\bibitem {bbh} A. Banerjee, S. Beyarslan, H. T. H${\rm \grave{a}}$, Regularity of powers of edge ideals: from local properties to global bounds, {\it Algebraic Combinatorics} {\bf 3} (2020), 839--854.

\bibitem {bht} S. Beyarslan, H. T. H${\rm \grave{a}}$, T. N. Trung, Regularity of powers of forests and cycles, {\it J. Algebraic Combin.} {\bf 42} (2015), 1077--1095.

\bibitem {bhz} M. Bigdeli, J. Herzog, R. Zaare-Nahandi, On the index of powers of edge ideals, {\it Comm. Algebra}, {\bf 46} (2018), 1080--1095.

\bibitem {cw} K. Cameron, T. Walker, The graphs with maximum induced matching and maximum matchingthe same size, {\it Discrete Math.} {\bf 299} (2005), 49--55.

\bibitem {dhs} H. Dao, C. Huneke, J. Schweig, Bounds on the regularity and projective dimension of ideals associated to graphs, {\it J. Algebraic Combin.} {\bf 38} (2013), 37--55.

\bibitem {e2} N. Erey, Powers of ideals associated to ($C_4,2K_2$)-free graphs, {\it J. Pure Appl. Algebra} {\bf 223} (2019), 3071--3080.

\bibitem {ehhs} N. Erey, J. Herzog, T. Hibi, S. Saeedi Madani, Matchings and squarefree powers of edge ideals, {\it J. Combin. Theory, Ser. A} {\bf 188} (2022), 105585.

\bibitem {eh} N. Erey, T. Hibi, Squarefree powers of edge ideals of forests, {\it Electron. J. Combin.}, {\bf 28} (2021), no. 2, Research Paper P2.32.

\bibitem {f2} O. Favaron,  Very well covered graphs, {\it Discrete Math.} {\bf 42} (1982), 177--187.

\bibitem {hv} H. T. H${\rm \grave{a}}$, A. Van Tuyl, Monomial ideals, edge ideals of hypergraphs, and their graded Betti numbers, {\it J. Algebraic Combin.} {\bf 27} (2008), 215--245.

\bibitem {ha} H. T. H${\rm \grave{a}}$, Regularity of squarefree monomial ideals, In S.M. Copper and S. Sather-Wagstaff(Ed.) Connections Between Algebra, Combinatorics, and Geometry. Springer Proceedings in Mathematics  Statistics {\bf 76} (2014), 251--276.

\bibitem {he} J. Herzog, A generalization of the Taylor complex construction, {\it Comm. Algebra} {\bf 35} (2007), 1747--1756.

\bibitem {hh1} J. Herzog, T. Hibi, An upper bound for the regularity of powers of edge ideals, {\it Math. Scand.} {\bf 126} (2020), 165--169.

\bibitem {hhko} T. Hibi, A. Higashitani, K. Kimura, A. B. O'Keefe, Algebraic study on Cameron-Walker graphs, {\it J. Algebra} {\bf 422} (2015), 257--269.

\bibitem {js} A. V. Jayanthan, S. Selvaraja, Upper bounds for the regularity of powers of edge ideals of graphs, {\it J. Algebra} {\bf 574} (2021),
    184--205.

\bibitem {jns} A. V. Jayanthan, N. Narayanan, S. Selvaraja, Regularity of powers of bipartite graphs, {\it J. Algebraic Combin.}, {\bf 47} (2018), 17--38.

\bibitem {k} M. Katzman, Characteristic-independence of Betti numbers of graph ideals, {\it J. Combin. Theory, Ser. A} {\bf 113} (2006), 435--454.

\bibitem {km} G. Kalai, R. Meshulam, Intersections of Leray complexes and regularity of monomial
    ideals, {\it J. Combin. Theory Ser. A} {\bf 113} (2006), 1586--1592.

\bibitem {msy} M. Moghimian, S. A. Seyed Fakhari, S. Yassemi, Regularity of powers of edge ideal of whiskered cycles, {\it Comm. Algebra}, {\bf 45} (2017), 1246--1259.

\bibitem {p'} I. Peeva, {\it Graded syzygies}, Algebra and Applications, vol. 14, Springer-Verlag London Ltd., London, 2011.

\bibitem {sy} S. A. Seyed Fakhari, S. Yassemi, Improved bounds for the regularity of edge ideals of graphs, {\it Collect. Math.} {\bf 69} (2018), 249--262.

\bibitem {sy2} S. A. Seyed Fakhari, S. Yassemi, Improved bounds for the regularity of powers of edge ideals of graphs, {\it J. Commut. Algebra}, to appear.

\bibitem {vv} A. Van Tuyl, R. Villarreal, Shellable graphs and sequentially Cohen-Macaulay bipartite graphs, {\it J. Combin. Theory, Ser. A} {\bf 115} (2008), 799--814.

\bibitem {wo}  R. Woodroofe, Matchings, coverings, and Castelnuovo-Mumford regularity, {\it J. Commut. Algebra} {\bf 6} (2014), 287--304.

\end{thebibliography}
\end{document}